\documentclass[11pt]{article}
\usepackage{amssymb,amsthm,amsmath,latexsym}
\newcommand{\bbrack}[1]{{
  \mathchoice
    {\left\lbrack\!\!\left\lbrack #1 \right\rbrack\!\!\right\rbrack} 
    {\left\lbrack\!\left\lbrack #1 \right\rbrack\!\right\rbrack} 
    {} 
    {} 
  }
}
\usepackage{ stmaryrd }
\usepackage{authblk}
\usepackage[symbol]{footmisc}
\usepackage{graphics}
\usepackage{color}
\usepackage{graphicx}

\newtheorem {theorem}{Theorem}
\newtheorem {corollary}{Corollary}
\newtheorem {lemma}{Lemma}

\newtheorem {example}{Example}

\begin{document}
\title{Stirling numbers with higher level and records}
\author[1*]{F. A. Shiha}
\affil[*]{Correspondence: fshiha@mans.edu.eg}

\affil[1]{Department of Mathematics, Faculty of Science, Mansoura University, Mansoura, Egypt}

\date{}
\maketitle
\begin{abstract}
 In this present paper, we show that the Stirling numbers of the first kind with higher level connected with the probability distribution of the number of records and record times in the so-called $F^{\alpha}$-scheme. In addition, we determine the location of the maximum of the Stirling numbers of the first kind with higher level.
\end{abstract}
\maketitle


\section {Introduction}
Given a positive integer $s$, the Stirling numbers of the first kind with higher level (level $s$), denoted by $\bbrack{ \begin{smallmatrix} n \\ k \end{smallmatrix}}_s$ were introduced by Tweedie \cite{twee}:
\begin{equation}\label{E:rcent1}
\prod_{i=0}^{n-1}(x+ i^s )=\sum_{k=0}^{n}\bbrack{ \begin{matrix} n \\ k \end{matrix}}_s x^k \qquad, n\geq 0.
\end{equation}
Notice that when $s = 1$, $\bbrack{\begin{smallmatrix} n \\ k \end{smallmatrix}}=(-1)^{n-k}s(n,k)$ are the (unsigned) Stirling numbers of the first kind \cite{comtet74}. For $s=2$, $\bbrack{\begin{smallmatrix} n \\ k \end{smallmatrix}}_2=(-1)^{n-k}\,t(2n,2k)$, where $t(n,k)$ are the central factorial numbers of the first kind given by
\begin{equation}\label{E:cent1}
x^{[n]}=\sum_{k=0}^{n}t(n,k)x^k, \qquad \textnormal{(see \cite{but89, kim2018, riod})},
\end{equation}
where $ x^{[n]}=x\left(x+\frac{n}{2}-1\right)\cdots \left(x-\frac{n}{2}+1\right),\; n \geq 1, \; x^{[0]}=1.$

Recall that $u(n,k)=t(2n,2k)$ are called the central factorial numbers of the first kind with even indices, see, e.g., \cite{gel, merca2016, shiha}.

The numbers ${\bbrack{\begin{smallmatrix} n \\ k \end{smallmatrix}}}_s$ satisfy the recurrence relation
\begin{equation}\label{re1}
  {\bbrack{\begin{matrix} n \\ k \end{matrix}}}_s={\bbrack{\begin{matrix} n-1 \\ k-1 \end{matrix}}}_s+(n-1)^s {\bbrack{\begin{matrix} n-1 \\ k \end{matrix}}}_s,\quad n, k\geq 1,
\end{equation}
with ${\bbrack{\begin{smallmatrix} 0 \\ 0 \end{smallmatrix}}}_s=1$ and  ${ \bbrack{\begin{smallmatrix} n \\ 0 \end{smallmatrix}}}_s={\bbrack{\begin{smallmatrix} 0\\ k \end{smallmatrix}}}_s= 0$  for all $n\geq 1$.

Recently in \cite{kom 2021, t.kom 2021}, combinatorial properties of the numbers ${\bbrack{\begin{smallmatrix} n \\ k \end{smallmatrix}}}_s$ are studied, also some relations with other sequences such as poly-Cauchy numbers with higher level, Bernoulli polynomials, and central factorial numbers are discussed.
In \cite{taka, taka2023}, the poly-Cauchy numbers with level $2$, and with higher level $s$ are introduced and studied as extensions of the original poly-Cauchy numbers.

The numbers $\bbrack{ \begin{smallmatrix} n \\ k \end{smallmatrix}}_s$ have the following special values (see \cite{kom 2021}):
\begin{equation}\label{s1}
 \bbrack{ \begin{matrix} n \\ 1 \end{matrix}}_s=\left((n-1)!\right)^s,
\end{equation}
\begin{equation}\label{s2}
 \bbrack{ \begin{matrix} n \\ 2 \end{matrix}}_s=\left((n-1)!\right)^s\, H_{n-1}^{(s)},
\end{equation}
\begin{equation}
 \bbrack{ \begin{matrix} n \\ 3 \end{matrix}}_s=\left((n-1)!\right)^s\, \sum_{k=1}^{n-2}\, \frac{H_k^{(s)}}{(k+1)^2},
\end{equation}
where $H_n^{(k)}$ are the generalized harmonic numbers of order $k$ defined by
\begin{equation}
H_n^{(k)}=\sum_{\ell=1}^{n}\frac{1}{\ell^k},\quad n > 0,
\end{equation}
with $H_0^{(k)}=0$, and $H_n=H_n^1$ are the harmonic numbers with $H_0=0$.

In the following, we have the first few values of the array $\left(\bbrack{\begin{smallmatrix} n \\ k \end{smallmatrix}}_s\right)_{n,k\geq 0}$ for $0 \leq n, k \leq 5$.
{\footnotesize
\[
\left(
\begin{matrix}
  &1             & 0        &0          &0     &0 &0       \\
  & 0            & 1        & 0         & 0    &0 &0        \\
  & 0            &1        &1          & 0   &0   &0         \\
  & 0            &2^s       & 1+2^s     & 1  &0       &0           \\
  &0             &6^s       &2^s+3^s+6^s   &1+2^s+3^s  &1   &0   \\
  &0             &{24}^s    &6^s+8^s+12^s+24^s  &2^s+3^s+4^s+6^s+8^s+12^s &1+2^s+3^s+4^s  &1
\end{matrix}
\right).
\]}
As there seems to be no symmetry in any row, it is natural to ask where the maximum of each row is. That is, for fixed $n$, what is the value (or
values) of $k$ for which ${\bbrack{\begin{smallmatrix} n \\ k \end{smallmatrix}}}_s$ is maximal. In the present paper, we will answer this question. Also, it is worth to noting that, the classical Stirling numbers and many of their generalizations are very closely connected with records (see \cite{Arnold, Bala, Char2007a, Char2007b, Char2007c, sib}). Similarly, we will show that ${\bbrack{\begin{smallmatrix} n \\ k \end{smallmatrix}}}_s$ describe distributions of some record statistics in the $F^{\alpha}$-scheme.
\section{Some preliminaries}
In this section, we introduce the concepts of the classical record scheme and record in $F^\alpha$-scheme.
Let $X_1, X_2, \ldots $ be a sequence of independent random variables. The random variable $X_j$ is called a record (upper record) if $X_j>X_i$ for all $i=1,2, \ldots, j-1$ (by convention $X_1$ is a record). Let $N_n$ be the number of records among $X_1, X_2, \ldots , X_n$, $n=1,2, \ldots$. And let $L_n$ be the record times, $n=1,2, \ldots$. Then $ L_1=1$ with probability one, and
\begin{equation}\label{rtime}
L_{n+1}=min\{j :X_j>X_{L_n}\},\quad n\geq 1,
\end{equation}
and $X_{L_n}$ the corresponding record values.

The distributions of $N_n$ and $L_n$ are obviously distribution free since they depend
only on the ranks in the sequence $X_1, X_2, \ldots$.

R\'enyi \cite{reny} obtained the following result.
\begin{lemma} Let $X_1, X_2, \ldots$ be independent identically distributed
random variables with a common continuous distribution function $F(x)$ and let
\begin{equation*}
  I_i=\begin{cases}
        1, & \mbox{if } X_i>max\{X_1, X_2, \ldots,X_{i-1}\}, \mbox{ i.e. } X_i \mbox{ is a record }\\
        0, & \mbox{otherwise}.
      \end{cases}
\end{equation*}
Then the indicator variables $I_1, I_2, \ldots$ are independent with
\[
P(I_i=1)=1-P(I_i=0)=\frac{1}{i}, \quad i=1,2,\ldots.
\]
\end{lemma}
From this result and the fact that $N_n=\sum_{i=1}^{n}I_i$, R\'enyi \cite{reny} obtained the probability functions of $N_n$ and $L_n$ in terms of the signless Stirling numbers of the first kind.

Nevzorov \cite{nev84, nev86, nev87, nev88} suggested and discussed records in a sequence of independent nonidentically distributed random variables. A generalization of the so-called $F^\alpha$-scheme. Deheuvels and Nevzorov \cite{deh1993} obtained various limit theorems (the central limit theorem, the laws of large numbers, the law of the iterated logarithm) for random variables $N_n$ and $L_n$ in $F^\alpha$-scheme  .

Let $X_1, X_2, \ldots$ be a sequence of independent random variables with distribution functions $F_1(x), F_2(x), \ldots$, we say that this sequence is an $F^\alpha$-scheme if
\begin{equation}
F_{i}(x)=[F(x)]^{\alpha_i},\quad i=1, 2, \ldots,
\end{equation}
where $F(x)$ is an arbitrary continuous distribution function and $\alpha_1, \alpha_2, \ldots$ are positive constants (without loss of generality, assume that $\alpha_1=1$).

\begin{lemma}\textnormal{(Nevzorov \cite{nev86, nev87})}
 In the $F^\alpha$-scheme, the record indicators $I_1, I_2, \ldots$ , which are defined as in Lemma $1$ are also independent and
\begin{equation}\label{inde}
  P(I_i=1)=1-P(I_i=0)=\frac{\alpha_i}{\alpha_1+\alpha_2+\cdots +\alpha_i}, \quad n=1,2,\ldots.
\end{equation}
\end{lemma}
Thus, for an $F^{\alpha}$-scheme, the number of records $ N_n $ can also be represented as a sum $I_1+I_2+\cdots +I_n$ of independent indicators. Balakrishnan and Nevzorov \cite{Bala} expressed the probability function of $N_n$ and $L_n$ in terms of generalized signless Stirling numbers of the first kind. In a geometrically increasing population, Charalambides \cite{Char2007c} expressed the probability functions and binomial moments of $N_n$ and $L_n$ in terms of the unsigned $q$-Stirling numbers of the first kind. Ahsanullah \cite{ahs2017} showed how to express distributions of records based on sequences of nonidentically distributed exponential random variables via distributions of independent random summands.
\section{ Distributions of the number of records and record times}
In this section, we obtain the probability function of the number of records $N_n$ and the record times $L_n$ in terms of the Stirling numbers of the first kind with level $s$.

Let us introduce independent random variables $I_i$, $i=1, 2, \ldots$ each of them taking values $0$ and $1$ with probabilities
\begin{equation}\label{E:ind}
 p_i= P(I_i=1) =1-P(I_i=0) =\frac{1}{1+(i-1)^s},\quad i=1,2,\ldots.
\end{equation}
Consider a sequence of independent but not identically random variables $X_1, X_2,\ldots $ have the distribution functions $F(x)^{\alpha_1}, F(x)^{\alpha_2}, \ldots$, respectively, where $F(x)$ is an arbitrary continuous function, while $\alpha_i$ are positive numbers, and
\begin{equation}\label{ind}
  p_i=P(I_i=1)=\frac{\alpha_i}{\alpha_1+\alpha_2\cdots +\alpha_i},\quad i=1,2,\ldots.
\end{equation}
The sequence $\{\alpha_i\}_{i=1}^\infty$ are determined conversely from $p_i, i=2,3, \ldots$ by
\begin{equation}
 \alpha_i=(\alpha_1+\alpha_2+\cdots +\alpha_{i-1})\frac{p_i}{1-p_i},\quad i=2,3,\ldots.
\end{equation}
 Starting from $\alpha_1=1$, we get
 \[
  \alpha_2=1,\quad \alpha_3=\frac{2}{2^s}, \quad \alpha_4=\frac{2(1+2^s)}{2^s(3^s)},\quad \alpha_5=\frac{2(1+2^s)(1+3^s)}{2^s(3^s)(4^s)}
 \]
Consequently, we obtain the coefficients of the $F^\alpha$-scheme in the form
\begin{equation}\label{alpha}
 \alpha_1=1,\quad \alpha_i=\prod_{j=2}^{i}\frac{(1+(j-2)^s)}{(j-1)^s}, \quad i=2,3,\ldots.
\end{equation}
Hence
\begin{equation}\label{alpha1}
  \alpha_{i+1}=\prod_{j=2}^{i+1}\frac{(1+(j-2)^s)}{(j-1)^s}=\alpha_i \,\frac{1+(i-1)^s}{i^s},
\end{equation}
then the sequence $\{\alpha_i\}_{i=1}^\infty$ satisfies the first order recurrence
\[
i^s \,\alpha_{i+1}=(1+(i-1)^s)\,\alpha_i \quad \Rightarrow \alpha_i=i^s \, \alpha_{i+1}-(i-1)^s\, \alpha_i, \quad \alpha_1=1.
\]
Summing for $i=1, 2,\ldots,j$ yields,
\begin{equation}\label{alphaj}
\sum_{i=1}^{j} \alpha_i=\sum_{i=1}^{j}\big( i^s \alpha_{i+1}-(i-1)^s \alpha_i\big)=j^s\,\alpha_{j+1}.
\end{equation}
Let us denote $A_0=0$,  $A_j=\sum_{i=1}^{j}\alpha_i$, then $A_j$ is a strictly increasing function and from (\ref{alphaj}), we get
\begin{equation}
\frac{A_j}{\alpha_{j+1}}=j^s, \quad j \geq 1.
\end{equation}
Since the number of records $N_n$ is obtained by $N_n=\sum_{i=1}^{n}I_i$, then the probability generating function of $N_n$ will be the product of the $n$ generating functions of the Bernoulli random variables $I_1, I_2, \ldots, I_n$,
 \begin{equation*} \label{E:pd}
 \begin{split}
 \textnormal{E}(t^{N_n})&= \prod_{i=1}^{n} (1-p_i+p_it)=\prod_{i=1}^{n}p_i\prod_{i=1}^{n}\left(t+\frac{1-p_i}{p_i}\right)\\
&=\prod_{i=1}^{n}p_i\prod_{i=1}^{n}\left(t+\frac{\alpha_1+\alpha_2+\cdots +\alpha_{i-1}}{\alpha_i}\right)\\
&=\prod_{i=1}^{n}p_i\prod_{i=1}^{n}\left(t+\frac{A_{i-1}}{\alpha_i}\right)
=\prod_{i=1}^{n}p_i\prod_{i=1}^{n}(t+(i-1)^s)\\
&=\prod_{i=1}^{n}p_i \, \sum_{k=0}^{n} {\bbrack{\begin{matrix} n \\ k \end{matrix}}}_s \,t^k
=\sum_{k=0}^{n}\, t^k\,{\bbrack{\begin{matrix} n \\ k \end{matrix}}}_s \, \prod_{i=1}^{n}p_i
\end{split}
\end{equation*}
Using the fact that $\textnormal{E}(t^{N_n})=\sum_{k=0}^{n} t^k P(N_n=k)$, we obtain
\begin{equation}\label{pn}
  P(N_n=k)= {\bbrack{\begin{matrix} n \\ k \end{matrix}}}_s \,\prod_{i=1}^{n}p_i
  =\frac{{\bbrack{\begin{smallmatrix} n \\ k \end{smallmatrix}}}_s}{\prod_{i=1}^{n}(1+(i-1)^s)}, \qquad 1\leq k \leq n.
\end{equation}
The expected value (mean) of the number of records $N_n$ can be obtained by
\begin{equation}\label{mean1}
 \textnormal{E}(N_n)=\sum_{k=1}^{n}k P(N_n=k)=\frac{1}{\prod_{i=1}^{n}(1+(1-i)^s)} \,\sum_{k=1}^{n}k\,{\bbrack{\begin{matrix} n \\ k \end{matrix}}}_s.
\end{equation}
Now the probability function of $L_k$ may be deduced from the probability function of
the number $N_n$ of records up to time $n$.  Since the following equality hold for any $n, = 1, 2, \ldots$ and $k, = 1, 2, \ldots$
\[
P(L_k=n)= P(N_{n-1}=k-1, N_n=k)=P(N_{n-1}=k-1, I_n=1),
\]
where the events $ \{N_{n-1}=k-1 \} $ and $ \{I_n=1 \}$ are independent in the $F^{\alpha}$-scheme, it follows from (\ref{pn}) that
\begin{equation}\label{pn2}
\begin{split}
 P(L_k=n)&= P(N_{n-1}=k-1)P(I_n=1)\\
 &= \frac{{\bbrack{\begin{smallmatrix} n-1 \\ k-1 \end{smallmatrix}}}_s}{\prod_{i=1}^{n}(1+(i-1)^s)},\quad n=k,k+1,\ldots.
\end{split}
\end{equation}
\section{The Maximum of ${\bbrack{\begin{smallmatrix} n \\ k \end{smallmatrix}}}_s$}
Determining the maximum value of Stirling numbers is an interesting problem to consider. The Stirling numbers form a strongly log-concave sequence \cite{harp, lieb}. Because of the log-concavity of Stirling numbers, it is enough to locate a maximum \cite{erd, ham}. Mez\"o \cite{mezo} gave estimations of the largest index for which the sequence of $r$-Stirling numbers assume its maximum. Corcino et al. \cite{core} established an estimating index at which the maximum value of the sequence of noncentral Stirling numbers of the first kind occurs.

In this section, we determine the location of the maximum of the Stirling numbers of the first kind with level $s$ and the corresponding maximum value.

We say that the sequence of real numbers $\{a_{i}\}_{i=1}^{n}$ is
unimodal if there exists an index $k$ such that $1 \leq k \leq n$, and $a_1\leq a_2 \leq \cdots \leq a_k \geq a_{k+1} \geq \cdots \geq a_n $.
A stronger property, called log-concavity, implies the unimodality. The sequence $\{a_{i}\}_{i=1}^{n}$  is called log-concave if
\begin{equation}\label{log}
  (a_k)^2 \geq a_{k-1}\, a_{k+1},  \qquad \textnormal{for} \; k=2, 3, \ldots, n-1.
\end{equation}
The sequence is strictly log-concave if we replace the $ ``\geq"$ by $``>"$ in (\ref{log}). The following theorem of Newton gives a useful strategy for proving the strict log-concavity, see \cite{bona, lieb, wilf}.
\begin{theorem} \label{tlog}
If the polynomial $ \sum_{k=1}^{n}a_k x^k$ has only real roots, then the sequence $\{a_{k}\}_{k=1}^{n}$ is strictly log-concave and it is also unimodal.
\end{theorem}
 For any fixed positive integer $n$, the polynomial given in (\ref{E:rcent1}) has only nonpositive real roots. Theorem \ref{tlog} immediately yields that the coefficient sequence $\{{\bbrack{ \begin{smallmatrix} n \\ k \end{smallmatrix}}}_s\}_{k=0}^n$ is strictly log-concave:
\begin{equation}
  \left({\bbrack{\begin{matrix} n \\ k \end{matrix}}}_s\right)^2-{\bbrack{\begin{matrix} n \\ k-1 \end{matrix}}}_s \; {\bbrack{\begin{matrix} n \\ k+1 \end{matrix}}}_s >0,\quad k=1,2,\ldots  ,n-1,
\end{equation}
and therefore unimodal. According to Hammersley \cite{ham} and Erdos \cite{erd} there exist a unique index $k_{n,s}$ of the maximal Stirling number of the first kind with level $s$ for all fixed  $n\geq3$
\[
{\bbrack{\begin{matrix} n \\ 1 \end{matrix}}}_s < {\bbrack{\begin{matrix} n \\ 2 \end{matrix}}}_s < \cdots < {\bbrack{\begin{matrix} n \\ k_{n,s}-1 \end{matrix}}}_s <  {\bbrack{\begin{matrix} n \\ k_{n,s} \end{matrix}}}_s >{\bbrack{\begin{matrix} n \\ {k_{n,s}+1} \end{matrix}}}_s > \cdots > {\bbrack{\begin{matrix} n \\ n \end{matrix}}}_s.
\]
To find the maximizing index, we use the following theorem of Darroch \cite{bona, dar}.
\begin{theorem}\label{mod}
  Let $G(x)=\sum_{k=0}^{n}a_k\,x^k$ be a polynomial that has real roots only
that satisfies $G(1)>0$. Let $k_n$ be the maximizing index for the sequence of the coefficients of $G(x)$. Let $\mu = G'(1)/G(1)$.
Then we have
\begin{equation}
|k_n-\mu|<1.
\end{equation}
\end{theorem}
Let $G_s(x)=\sum_{k=0}^{n}{\bbrack{\begin{smallmatrix} n \\ k \end{smallmatrix}}}_s \, x^k$ and $\mu_{n,s}=G'(1)/G(1)$.
So if we can compute $\mu_{n,s}$, we need to check at most two values of ${\bbrack{\begin{smallmatrix} n \\ k \end{smallmatrix}}}_s$ to find the maximum of the sequence $\{{\bbrack{\begin{smallmatrix} n \\ k \end{smallmatrix}}}_s\}_{k=1}^n$. In the next theorem, we compute $\mu_{n,s}$, and then we can locate (up to distance $1$) the maximum of the sequence $\{{\bbrack{\begin{smallmatrix} n \\ k \end{smallmatrix}}}_s\}_{k=1}^n$ for any fixed $n>2$.
\begin{theorem}
For any fixed positive integer $n >2$, let $k_{n,s}$ be the largest index for which the sequence $\{{\bbrack{\begin{smallmatrix} n \\ k \end{smallmatrix}}}_s\}_{k=1}^n$
 assumes its maximum, then
\begin{equation}
  \left|k_{n,s}-\sum_{i=1}^{n}\frac{1}{1+(i-1)^s}\right|<1.
\end{equation}
\end{theorem}
\begin{proof} For $n>2$, we have
\[
G_s(1)=\sum_{k=0}^{n}{\bbrack{\begin{matrix} n \\ k \end{matrix}}}_s=\prod_{i=1}^{n}\left(1+(i-1)^s\right)>0,
\]
and
\[
{G'}_s(x)=\sum_{k=0}^{n} k \,{\bbrack{\begin{matrix} n \\ k \end{matrix}}}_s x^{k-1}\quad \Rightarrow \quad {G'}_s(1)=\sum_{k=0}^{n} k\,{\bbrack{\begin{matrix} n \\ k \end{matrix}}}_s.
\]
To compute ${G'}_s(1)$, we notice that $\textnormal{E}(N_n)$ which given in (\ref{mean1}) can be also given in an equivalent form by the fact that, the random variable $N_n$ is represented as a sum of independent nonidentically Bernoulli random variables. Then the random variable $N_n$ has a Poisson-binomially distribution with mean given by
\begin{equation}{\label{mean2}}
\textnormal{E}(N_n)=\textnormal{E}(\sum_{i=1}^{n}I_i)=\sum_{i=1}^{n}\textnormal{E}(I_i)
=\sum_{i=1}^{n} p_i=\sum_{i=1}^{n}\frac{1}{1+(i-1)^s},
\end{equation}
From (\ref{mean1}) and (\ref{mean2}), we obtain
\[
\frac{1}{\prod_{i=1}^{n}(1+(1-i)^s)} \,\sum_{k=1}^{n}k\,{\bbrack{\begin{matrix} n \\ k \end{matrix}}}_s =\sum_{i=1}^{n} \frac{1}{1+(i-1)^s}.
\]
Hence, we get
\begin{equation}
G'(1)=\sum_{k=1}^{n}k\,{\bbrack{\begin{matrix} n \\ k \end{matrix}}}_s =\prod_{i=1}^{n}\left(1+(1-i)^s\right)\sum_{i=1}^{n}\frac{1}{1+(i-1)^s}.
\end{equation}
Then
\begin{equation}\label{mu}
  \mu_{n,s}=\frac{G'(1)}{G(1)}=\sum_{i=1}^{n}\frac{1}{1+(i-1)^s}.
\end{equation}
\end{proof}
Notice that the values of $\mu_{n,s}$ can be easily computed for any $n$ and $s$, and the unique value of $k_{n,s}$ for which ${\bbrack{\begin{smallmatrix} n \\ k_{n,s} \end{smallmatrix}}}_s$ is maximal is within distance $1$ of $\mu_{n,s}$.

\textbf{Some special cases}
Using the recurrence (\ref{re1}) and any computer algebra system, one can easily obtain the values of the arrays ${\bbrack{\begin{smallmatrix} n \\ k \end{smallmatrix}}}_s$ for any $n$ and $s$.

\textbf{Case 1:} When $s=1$, we have ${\bbrack{\begin{smallmatrix} n \\ k \end{smallmatrix}}}_1=|s(n,k)|$, the signless Stirling numbers of the first kind, in this case $\mu_{n,1}=\sum_{i=1}^{n}\frac{1}{i}=H_n$, which coincide with \cite{bona} result.

\begin{example}
  For $n=6$, we get $\mu_{6,1}=2.45$ and hence the maximizing index of the sequence $\{|s(6,k)|\}_{k=1}^6$  satisfies $|k_{6,1}-2.45|<1$, then $k_{6,1}=2$ or $3$. Therefore, we check only the two values, $|s(6,2)|$ and $|s(6,3)|$. In fact $|s(6,2)|=274$ is the maximal.
\end{example}
\begin{example} For $n=100$, we have $\mu_{100,1}=5.187$, then $k_{100,1}=5$ or $6$. Indeed $|s(100,5)|=1.971\cdot 10^{157}$ is the maximum of the sequence $\{|s(100,k)|\}_{k=1}^{100}$.
\end{example}

\textbf{Case 2:} When $s=2$, we have ${\bbrack{\begin{smallmatrix} n \\ k \end{smallmatrix}}}_2=(-1)^{n-k}\, u(n,k)$, where $u(n,k)$ are the central factorial numbers of the first kind with even indices.

\begin{example}
For $n=5$, we have $\mu_{5,2}=\sum_{i=1}^{5}\frac{1}{1+(i-1)^2}=1.859$, and the maximizing index of the sequence $\{{\bbrack{\begin{smallmatrix} 5 \\ k \end{smallmatrix}}}_2\}_{k=1}^5$ satisfies $|k_{5,2}-1.859|<1$, then $k_{5,2}=1$ or $2$.  Therefore, we check only the values ${\bbrack{\begin{smallmatrix} 5 \\ 1 \end{smallmatrix}}}_2$ and ${\bbrack{\begin{smallmatrix} 5 \\ 2 \end{smallmatrix}}}_2$. In fact, ${\bbrack{\begin{smallmatrix} 5 \\ 2 \end{smallmatrix}}}_2=820$ is the maximal.
\end{example}
\begin{example}
For $n=100$, we have $\mu_{100,2}=2.067$, then $k_{100,2}=2$ or $3$. Indeed ${\bbrack{\begin{smallmatrix} 100 \\ 2 \end{smallmatrix}}}_2=1.4239\cdot10^{312}$ is the maximum of the sequence $\{{\bbrack{\begin{smallmatrix} 100 \\ k \end{smallmatrix}}}_2\}_{k=1}^{100}$.
\end{example}
Notice that, $\mu_{n,2}$ increases slowly with $n$, its minimum value is $1.7$ and its maximum is  $2.074$. For any fixed $n\geq3$, we find that the maximum of the sequence  $\{{\bbrack{\begin{smallmatrix} n \\ k \end{smallmatrix}}}_2\}_{k=1}^{n}$ occurs at  $k=2$.

\textbf{Case 3:} When $s=3$, some values of ${\bbrack{\begin{smallmatrix} n \\ k \end{smallmatrix}}}_3$ are illustrated in Table 1 for $3\leq n\leq 12$.
{\footnotesize
\begin{table}[!ht]
\centering
\caption{Values of ${\bbrack{\begin{smallmatrix} n \\ k \end{smallmatrix}}}_3$ for $3\leq n\leq 12$}\label{tab2}
\begin{tabular}[p]{| l|c | c| c| c| c| c| c| c| } \hline \hline
$n\,/\,k$ &1  &2  &3   &4    &5    &6         \\ \hline
3 &8         &9        &1         &0        &0           &0                 \\ \hline
4 &216       &251      &36        &1        &0           &0                 \\ \hline
5&13820     &16280    &2555      &100      &1           &0                 \\
6&$1.728\cdot10^6$    &$2.049\cdot10^6 $       &$3.357\cdot10^5 $      &15060     &225 &1     \\ \hline
7&$3.732\cdot10^8$   &$4.443\cdot10^8 $   &$7.455\cdot10^7 $  &$3.588\cdot10^6$    &63660&     441       \\ \hline
8 &$1.280\cdot10^{11}$   &$1.528\cdot10^{11}$     &$2.602\cdot10^{10}$    &$1.305\cdot10^9$ & $2.542\cdot10^7$  &$ 2.149\cdot10^5 $  \\ \hline
9 &$6.555\cdot10^{13}$&$7.834\cdot10^{13}$ &$1.347\cdot10^{13}$ &$6.942\cdot10^{11}$ &$1.432\cdot10^{10}$ & $ 1.355\cdot10^8 $ \\ \hline
10 &$4.778\cdot10^{16 }$&$5.718\cdot10^{16}$ &$9.900\cdot10^{15}$ &$5.196\cdot10^{14}$ &$1.113\cdot10^{13}$ &$1.131\cdot10^{11}$ \\ \hline
11 & $4.778\cdot10^{19}$ &$ 5.722\cdot10^{19}$& $9.957\cdot10^{18}$ &$5.295\cdot10^{17}$ &$1.165\cdot10^{16 }$&$1.242\cdot10^{14}$  \\ \hline
12 &$6.360\cdot10^{22}$ & $7.621\cdot10^{22}$ &$1.331\cdot10^{22}$ &$7.147\cdot10^{20}$&$1.604\cdot10^{19}$ &$1.770\cdot10^{17}$  \\ \hline
\end{tabular}
\end{table}}
It is shown that the maximum of each row of the sequence  $\{{\bbrack{\begin{smallmatrix} n \\ k \end{smallmatrix}}}_3\}$, $ 3 \leq n\leq 12$ occurs at index $k=2$. In fact, we obtain the same maximizing index for larger values of $n$.

It is worth noting that for fixed $n\geq 3$, we have $\mu_{n,2}>\mu_{n,3}>\mu_{n,4}>\cdots$. In Table 2, the minimum and the maximum of $\mu_{n,s}$ for $3 \leq s \leq 8$ are given.
\begin{table}[!ht]
\centering
\caption{Values of the minimum and the maximum of $\mu_{n,s}$ for $3 \leq s \leq 8$ }\label{tab2}
\begin{tabular}{| c|c | c| c| c| c| c|} \hline
 &$\mu_{n,3}$ &$\mu_{n,4}$  & $\mu_{n,5}$ & $\mu_{n,6}$ &$\mu_{n,7}$ & $\mu_{n,8}$   \\ \hline
\textnormal{Minimum}   &1.611 &1.559  &1.530 &1.515  &1.508 &1.504 \\
 \textnormal{Maximum}   &1.686  &1.578  &1.536  &1.517  &1.508  &1.504  \\ \hline
\end{tabular}
\end{table}

In fact, we can see that $\mu_{n,s}=1.5+\sum_{i=2}^{n-1} \frac{1}{1+i^s}$ is converging to the number $1.5$ as $s$ increasing. So, the value of the maximizing index must be equal to $1$ or $2$.  According to (\ref{s1}) and (\ref{s2}), we have ${\bbrack{\begin{smallmatrix} n \\ 1 \end{smallmatrix}}}_s <{\bbrack{\begin{smallmatrix} n \\ 2 \end{smallmatrix}}}_s$, therefore we get the following result.
\begin{corollary}
For any fixed integer $n\geq 3$ and any integer $s\geq 2$, the maximum value of the sequence $\{{\bbrack{\begin{smallmatrix} n \\ k \end{smallmatrix}}}_s\}_{k=1}^n$ is given by ${\bbrack{\begin{smallmatrix} n \\ 2 \end{smallmatrix}}}_s=\left((n-1)!\right)^s\, H_{n-1}^{(s)} $.
\end{corollary}

\section{Conclusions}
In this paper, we expressed the probability function and the mean of the number of upper records $N_n$ in terms of the Stirling numbers of the first kind with level $s$. Consequently, the probability function of the record time $L_n$ is also obtained in terms of the Stirling numbers of the first kind with level $s$.  Finally, we determined the location of the maximum of the Stirling numbers of the first kind with level $s$.

\makeatletter
\renewcommand{\@biblabel}[1]{[#1]\hfill}
\makeatother

\end{document}